\documentclass{amsart}
\usepackage{hyperref}
\usepackage{amsthm}
\usepackage{amsmath}
\usepackage{amssymb}
\usepackage{xfrac}
\usepackage{mathtools}
\usepackage{tikz-cd}
\usepackage{dynkin-diagrams}
\usepackage{array}

\newtheorem{prop}{Proposition}
\newtheorem{definition}[prop]{Definition}
\newtheorem{theorem}[prop]{Theorem}

\newtheorem{cor}[prop]{Corollary}
\newtheorem{lemma}[prop]{Lemma}

\newtheorem{rem}[prop]{Remark}
\newtheorem{example}[prop]{Example}

\numberwithin{prop}{section}

\newcommand{\R}{\ensuremath{\mathcal{R}}}

\newcommand{\LL}{\ensuremath{\mathcal{L}}}

\newcommand{\BR}{\ensuremath{\mathbb{R}}}

\newcommand{\On}[1]{\mathrm{O}_{#1}\left(\BR\right)}

\newcommand{\ba}{\sqrt{\frac{3}{4}-\cos^2(\pi/5)}}
\newcommand{\co}{\cos(\pi/5)}
\newcommand{\Rperp}[1]{\R_{#1}^{\perp}}
\newcommand{\si}{\sin{\pi/5}}

\DeclarePairedDelimiter\abs{\lvert}{\rvert}%
\DeclarePairedDelimiter\norm{\lVert}{\rVert}%

\makeatletter
\let\oldabs\abs
\def\abs{\@ifstar{\oldabs}{\oldabs*}}
\let\oldnorm\norm
\def\norm{\@ifstar{\oldnorm}{\oldnorm*}}
\makeatother

\title{Optimal Finite Homogeneous sphere approximation  }

\author{Omer Lavi}

\begin{document}

\begin{abstract}
The two dimensional sphere can't be approximated by finite homogeneous spaces. We describe the optimal approximation and its distance from the sphere.
\end{abstract}

\maketitle

\section{Introduction}

Tsachik Gelander and Itai Benjamini asked me (see also Remark 1.5 in \cite{gelander2012limits}) what is the optimal approximation of the sphere by a finite homogeneous space. Denote by $d_H$ the Hausdorff metric. We will say that a finite $\On{n}$-homogeneous space \(X\) is an optimal approximation of a set \(A\) if $d_H(X,A) \leq d_H(Y,A)$ for every finite $\On{n}$-homogeneous space \(Y\). In that case, we say that the approximation distance of \(A\) is $d_H(X,A)$. A set $A \subset \BR^n$ is approximable by finite homogeneous spaces if it is a limit of such. An n-dimensional torus, for example, can be approximated by a sequence of regular graphs with bounded geometry (see \cite {benjamini2012scaling}). Gelander showed that these are the only examples: a compact manifold can be approximated by finite metric homogeneous spaces if and only if it is a torus (Corollary 1.3 in \cite{gelander2012limits}).

 \begin{theorem} \label{thm:main result}
 	The approximation distance of the sphere (calculated up to 4 digits) is 0.3208.
 \end{theorem}
\begin{rem}
	Note that in this paper we restrict our attention to $\On{3}$ spaces and their Housdorff distances from the sphere. It is an interesting question if optimal finite homogeneous sphere approximation in the Gromov Housdorff sense exist and what are they. 
\end{rem}

In this paper we will prove this theorem and will describe explicitly the optimal approximation of the sphere.

\section {Finite Coxeter groups}
We are interested in finite homogeneous spaces of $\BR^3$. We wish then to understand the finite subgroups of the group of linear isometries of $\BR^3$, which we denote by $\On{3} $.

Coxeter described discrete subgroups of the group of isometries of $\BR^n$ in the enjoyable paper \cite{coxeter1940regular}. The exposition here is based on his description. These groups, known as Coxeter groups, are in general groups generated by reflections. There are also subgroup of reflection groups. We are interested, however, in maximal finite groups, which are always reflection groups. One can think of them as describing $n-1$ dimensional mirrors of a kaleidoscope. If we compose two reflections with angle $ \phi$ between them, we get a rotation of angle $2\phi$.
Note that if $G$ is a group generated by reflections and $G$ is finite, then the angles between any pair of reflections must be either $0$ or $\pi/n$ for some integer $n$. We adopt the notation that $\pi/\infty$ describes reflections about two parallel hyper planes. If the group is a subgroup of $\On{3}$ (instead of the complete isometry group) the reflection hyper planes are of course concurrent. In this case the sphere is left invariant under the action. The restriction of the action to the sphere gives an action on the sphere, $S^2$. We reflect then in great circles of the sphere. Note that if $R$ is a reflection related to a plane $\R<V$ in some vector space, then $\R$ is the space of invariant vectors of $R$. Similarly, if $H=\langle R_1,R_2 \rangle$ is the group generated by reflections related to $\R_1,\R_2$ respectively then $\R_1 \cap \R_2$ is the space of $H$-invariant vectors. We denote this space by $V^H$. In the case that $V=\BR^3$, and $H$ is generated by two reflections, $V^H \cap S^2$ has two antipodal points. Throughout the sequel, we will denote the reflecting plane related to a reflection $R$ by $\R$. We fix a unit vector at the orthogonal complement and denote it by $\R^{\perp}$.

The Coxeter groups can be described by diagrams called the Coxeter Dynkin diagrams, which we presently describe. The nodes of a diagram are related to reflections. Each node represents a reflection. We already mentioned that the angles between two planes has to be of the form $\pi / n$ for some (possibly infinite) integer $n$. Two vertices are connected if the angle between the reflecting planes is smaller than $\pi/2$ (this notation makes sense, otherwise the group is just a product of the two groups). In this case the number above the connecting edge is just $n$ (a common practice which we will adopt here, is to omit the number when $n=3$).
\begin{example} \label{exam: coxeter diagram}
The diagram
\dynkin[Coxeter]{B}{3} describes a group named $B_3$ (see more details below), which is generated by three reflections of mirrors with angles $\pi/3,\pi/4$ and $\pi/2$.
\end{example}
 Set $V= \BR^3$. Finite groups of $\On{3}$ are generated by three or two reflections. The volume enclosed by the reflection hyperplane is then a fundamental domain for the action of the group on $\BR^3$. In the sequel we will denote this fundamental domain by $B$. For the action on the unit sphere we have then a fundamental domain which is an intersection of $B$ and the sphere. We will denote this domain by $B'$. When the group is generated by three reflections, the fundamental domain is thus a spherical triangle. If $G=\langle R_1,R_2,R_3\rangle$ and we set $G_1=\langle R_1,R_2\rangle,\quad G_2 = \langle R_2,R_3\rangle,\quad G_3=\langle R_3,R_1\rangle $ then a choice of points in $V^{G_1}\cap S^2, \quad V^{G_2}\cap S^2, \quad V^{G_3}\cap S^2$ gives us the vertices of this triangle and a choice of respective segments in $V^{R_i} \cap S^2$ gives the sides of this triangle. We will therefore denote the groups $G_i$ as vertex groups and the groups $R_i$ as edge groups. Note that every group element in $G$ is conjugated to an element in a vertex groups.

A group related to such diagram can be described the by a notation $[m,n,\ldots ,k]$. As finite groups of $\On{3}$ can always be generated by two or three reflections, the Coxeter groups related to finite subgroup of $\On{3}$ are of the form $[m,n]$. The groups $[2,n]$ are defined for every integer $n$ while the groups $[3,n]$ can have $ n=3,4,5 $.
\begin{example}
	 The group $B_3$ described by the diagram in Example \ref {exam: coxeter diagram} above is denoted by $[3,4]$.
\end{example}
 A Coxeter group of the form $[m,n]$ can be presented as:
\begin {equation} \label{presentation of Coxeter groups}
\langle R_1,R_2,R_3 | R_1^2=R_2^2=R_3^2=(R_1R_2)^m=(R_2R_3)^n=(R_1R_3)^2=1 \rangle.
\end{equation}
 Note that the groups $[m,n]$ and $[n,m]$ are isomorphic.
 We add the notation
$$ S_1= R_1R_2 \qquad S_2=R_2R_3 \text{,  and}\qquad S_3=R_3R_1. $$
Then $S_1$ and $S_2$ generate a subgroup of index two. This group is the group of orientation preserving isometries or simply the group of all rotations. The element $S_1$ is a rotation  whose axis is simply the intersection $\R_1 \cap \R_2 = V^{G_1}$, while the rotation $S_2$ is a rotation whose axis is simply the intersection of the planes related to $R_2 \text{ and } R_3$. The angles of these rotations is $2\pi/m$ and $2\pi/n$ respectively. Note that $S_3=R_1R_3=S_1S_2$ is a $\pi$ rotation about the intersection of the planes related to $R_1 \text{ and }R_3$. The groups of rotations are denoted by $[m,n]'$ and can be presented as
$$\langle S_1,S_2| S_1^m=S_2^n=(S_1S_2)^2=1 \rangle. $$

First note the infinite family of (almost) cyclic groups. The groups $[2,n]$ has only one axis for the rotations. They have an index two cyclic subgroup. An orbit of a point under the action of such group is contained in at most two latitudes. Optimally these latitudes will be located at an angle $\pi/4$ from the equator (possibly on a sphere of shorter radius). In that case the distance from any point $x$ on the orbit to the north pole $n$ will be $d(x,n) =\frac {\sqrt{2}}{2}$. Therefore the Hausdorff distance between homogeneous spaces of these groups to the sphere is at least 0.7.

We are mainly interested in groups that have more than one axis for their rotations. The groups $[3,3],[3,4],[3,5]$, whose Coxeter diagrams are \dynkin[Coxeter]{A}{3},  \dynkin[Coxeter]{B}{3} and \dynkin[Coxeter]{H}{3} respectively, are denoted by $A_3, B_3,H_3$ respectively, and their rotational subgroups, $[3,3]',[3,4]',[3,5]'$, are called the \textit{tetrahedral, octahedral} and \textit{icosahedral} groups respectively. The group $[3,3]$ is isomorphic to the symmetric group $S_4$ and has 24 elements. One of its homogeneous spaces is the Platonic solid, tetrahedron. The group $[3,4]$ has 48 elements. It is acting transitively on the cube and on its dual, the octahedron.  Finally the group $H_3$ is isomorphic to the product of the alternating group on five elements, $A_5$ (not to be confused with the Coxeter symbol) and the group with two elements. It has 120 elements, and the icosahedron and its dual, the dodecahedron, are $H_3$-homogeneous spaces. Another notable $H_3$-homogeneous space is the Archimedean solid, the truncated icosahedron, known mainly for its roll in sports (see more details below). The truncated icosidodecahedron has 120 vertices. It has 62 faces, 30 out of them are squares, 20 are regular hexagons and 12 are regular decagons (or 10-gons). It is the orbit of the barycenter of the fundamental domain of the $H_3$.

\subsection{Orbits of Coxeter groups}
Recall that when a group \(G\) is acting on a space \(X\), we denote by $Gx$ the orbit of a point $x \in X$. We are now interested in the groups $[3,n]$. The orbit of a point under the action of these groups can be viewed as lighting a candle and taking all its images in the kaleidoscope generated by the mirrors we relate to reflections. Of course  different points will result in different orbits under the same group.

\begin {example}[The Cube] \label{example: the cube}
An easy example is the Coxeter group $A_1 \times A_1\times A_1$ whose diagram is simply \dynkin[Coxeter]{A}{1} \dynkin[Coxeter]{A}{1} \dynkin[Coxeter]{A}{1}. The group is generated by $3$ mirrors which are perpendicular to each other. We can assume then that they are perpendicular to each of the axes.

The fundamental domain for the action on the sphere is then described by the spherical triangle whose vertices are
$$\left\lbrace (1,0,0), (0,1,0),(0,0,1) \right\rbrace. $$
The orbit of the center of this triangle, i.e. the point $c=(1,1,1)/{3}$, is the vertices of a cube. On the other hand the orbit of the point $(1,0,0)$ has two points only. It is interesting to mention that the distance between the cube described here and the sphere is exactly $\sqrt{\frac{2}{3}}$ (since this is the distance from $c$ to the vertices of the fundamental domain). The distance between the two points orbit to the sphere can easily be seen to be $\sqrt{2}$.
\end{example}

Given a group \(G\) we would like to find the point whose orbit is closest to the sphere among all other \(G\) orbits.
\begin{definition}\label{def: best approximation}
	Given a group $G < \On{3}$ we say that $x \in \BR^3$ is an optimal \(G\) \textit{sphere approximation} if
	$$d_{H}(G x ,S^2) = \inf_{y\in \BR^3}\left\lbrace d_{H}(Gy,S^2)\right\rbrace. $$
\end{definition}
The search for optimal approximation can be limited to fundamental domain:
\begin{rem}\label{rem: enough to look for approximation inside a fundamental domain}

	If $B\subset \BR^3$ is a fundamental domain for the action of \(G\) then $x\in B$ is optimal \(G\) approximation if and only if
		$$d_{H}(G x ,S^2) = \inf_{y\in B}\left\lbrace d_{H}(Gy,S^2)\right\rbrace. $$
		
 \end{rem}
We can define a distance between groups and the sphere.
\begin{definition} \label {def: distance between a group to the sphere}
	If \(G\) is a group the we define the distance between \(G\) and the sphere as
	$$d_H(G,S^2)=\inf_{y\in \BR^3}\left\lbrace d_{H}(Gy,S^2)\right\rbrace. $$
\end{definition}
\begin{rem}\label{rem: best approximation is in fundemantal domain}
	Note that  if $B$ is a fundamental domain and $B'=B \cap S^2$ then
	$$d_H(G,S^2) = \inf_{x\in B}\sup_{y \in B'}\left\lbrace d(Gx,y)\right\rbrace. $$
\end{rem}

	The next result will provide a candidate for optimal approximation.	
	The following lemma is well known (see for example Lemma 2.2.7 in \cite{bekka2008kazhdan}).
	\begin{lemma}[Lemma of the center] \label{lem: lemma of the center}
		Let $V$ be a real or complex Hilbert space, and let $A \subset V$ be a non-empty bounded set of \(V\). Among all closed balls in $V$ containing \(A\), there exist a unique one with minimum radius.
	\end{lemma}
	\begin{definition}
		The center of the unique closed ball with minimal radius containing \(A\) in the previous lemma is called Chebychev center (or circum center) of \(A\).
	\end{definition}

	\begin{example}\label{exam: Chebychev center of triangle}
		Let $A$ be an Euclidean triangle. In this case finding the Chebychev center is an easy task. For obtuse or right triangle the center is just at the middle of the long side. Otherwise it is at the intersection of the perpendicular bisectors of the sides of the triangle.
	\end{example}
	
	\begin {rem} \label{rem: are we taking pointson sphere}
	Since the fundamental domain $B'$ is not convex the Chebychev center is not contained in $B'$. Usually it will not be on the unit sphere $S^2$ but inside the ball of radius one (see for example the cube from Example \ref{example: the cube}).
	
\end{rem}
	
	 \subsection{Fundamental domain of Coxeter groups}
	 We describe now the fundamental domains of the $3$ dimensional Coxeter groups.
	 In general if the group is defined by $[n,m]$ (in our case we will always have $n=3$ ), then we have $3$ planes. As the group $\On{3}$ is acting transitively on $2$-dimensional subspaces, we can assume without any loss of generality that one of the planes, which we denote by $\R_1$, is $XY$, i.e. that it is defined by the normal vector $\Rperp{1}=(0,0,1)$. Note that $\On{2}$ is embedded naturally in $\On{3}$ as the stabilizer of $\R_1$, and is acting transitively on one dimensional subspaces. We can assume then that $\R_3$ is the plane $XZ$. It is perpendicular to $\R_1$ and its normal is $\Rperp{3} = (0,1,0)$. (We number this plane by $\R_3$ to be consistent with the notation of the Coxeter diagrams.)
	 The last space, $\R_2$ whose unit normal we denote by $\Rperp{2}$ intersects the other two with angles of $\pi/3$ and $\pi/m$ respectively. Set $\Rperp{2}=(x,y,z)$ then
	 \begin{equation}\label{eq: formula for z and y}
	 \langle \Rperp{1},\Rperp{2} \rangle= z=\cos(\pi/3)=\frac{1}{2}
	 \qquad \langle \Rperp{2},\Rperp{3} \rangle= y=\cos(\pi/m).
	 \end{equation}
	 Next, since we restrict our attention to unit vectors we can compute $x$:
	 \begin{equation}\label{eq: formula for x}
	 1=\langle \Rperp{2},\Rperp{2} \rangle= x^2+\cos^2(\pi/m)+\cos^2(\pi/3).
	 \end{equation}
	 %
	
	 We can calculate now spherical triangles for the three groups.
	 \begin{enumerate}
	 	\item \textbf{The tetrahedral group $A_3$.} The group \dynkin[Coxeter]{A}{3} , also denoted by $[3,3]$ is generated by three reflections whose angles are $\pi/3,\pi/3$ and $\pi/2$. As explained above, there is a fundamental domain described by the following three unit normal vectors:
	 	\begin{enumerate}
	 		\item $\Rperp{1}= (0,0,1)$
	 		\item $\Rperp{2} = (\frac{1}{\sqrt{2}},\frac{1}{2},\frac{1}{2})$
	 		\item $\Rperp{3} = (0,1,0)$
	 	\end{enumerate}
	 	Now we wish to construct a spherical triangle. The intersection $\R_1 \cap \R_3$ is the \(X\) axis, we denote the first vertex $z=(1,0,0)$. Next $\R_1\cap \R_2 = \left\lbrace (t,-\frac{2}{\sqrt{2}}t,0) : t\in \BR \right\rbrace $. Therefore we set $x = (\frac{1}{\sqrt{3}}, -\frac{2}{\sqrt{6}},0)$. Finally by symmetry,  $y = (\frac{1}{\sqrt{3}},0, -\frac{2}{\sqrt{6}})$.

	 	\item \textbf{The octahedral group $B_3$.} The group \dynkin[Coxeter]{B}{3} , also denoted by $[3,4]$ is generated by three reflections whose angles are $\pi/3,\pi/4$ and $\pi/2$. As explained above, there is a fundamental domain described by the following three unit normal vectors:
	 	\begin{enumerate}
	 		\item $\Rperp{1}=(0,0,1)$
	 		\item $\Rperp{2}=(\frac{1}{2},\frac{1}{\sqrt{2}},\frac{1}{2})$
	 		\item $\Rperp{3}=(0,1,0)$
	 	\end{enumerate}
	 	Now we wish to construct a spherical triangle. The intersection $\R_1 \cap \R_3$ is the \(X\) axis, we denote the first vertex $z=(1,0,0)$. Next  $\R_1\cap \R_2 = \left\lbrace (t,-\frac{\sqrt{2}}{2}t,0) : t\in \BR \right\rbrace $. Therefore we set $x = (\frac{\sqrt{2}}{\sqrt{3}}, -\frac{1}{\sqrt{3}},0)$. Finally $\R_2\cap \R_3 = \left\lbrace (t,0,-t) : t\in \BR \right\rbrace $. Therefore we set $y = (\frac{1}{\sqrt{2}},0, -\frac{1}{\sqrt{2}})$.
	 	\begin{figure}\label{fig: f.d. octahedron}
	 		
	 		\includegraphics[width=3.5cm, height=2.5cm]{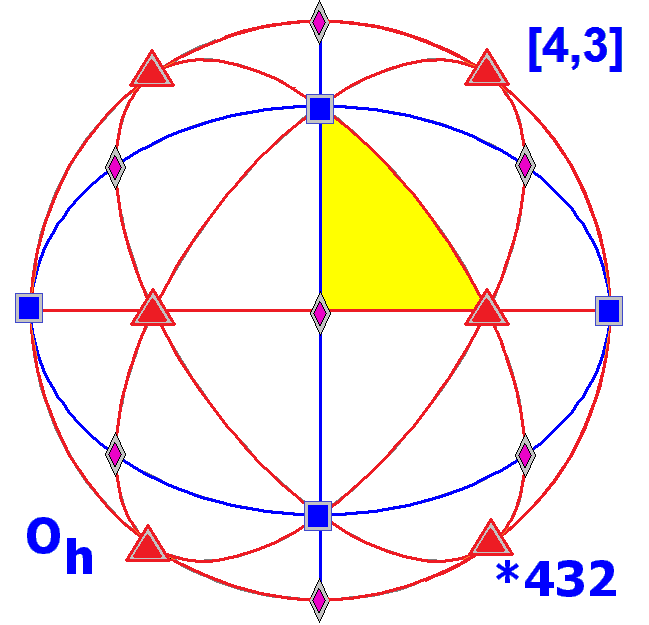}
	 		\caption{The octahedron group. A fundamental domain and symmetries. (Figure taken from Wikipedia)}
	 	\end{figure}
	 	
	 	\item \textbf{The icosahedral group $H_3$} The group \dynkin[Coxeter]{H}{3} , also denoted by $[3,5]$ is generated by three reflections whose angles are $\pi/3,\pi/5$ and $\pi/2$. As explained above, there is a fundamental domain described by the following three unit normal vectors:
	 	\begin{enumerate}
	 		\item $\Rperp{1}=(0,0,1)$
	 		\item $\Rperp{2}=(\sqrt{\frac{3}{4}-\cos^2(\pi/5)},\cos(\pi/5),\frac{1}{2})$
	 		\item $\Rperp{3}=(0,1,0)$
	 	\end{enumerate}
	 	Now we wish to construct a spherical triangle. The intersection $\R_1 \cap \R_3$ is the \(X\) axis, we denote the first vertex $z=(1,0,0)$. Next $\R_1\cap \R_2 = \left\lbrace (t\cos(\pi/5),-t\sqrt{\frac{3}{4}-\cos^2(\pi/5)},0) : t\in \BR \right\rbrace $. Therefore we set
	 	$$x = \frac{2}{\sqrt{3}}(\cos(\pi/5),-\sqrt{\frac{3}{4}-\cos^2(\pi/5)},0).$$
	 	Finally, $\R_2 \cap \R_3 = \left\lbrace(\frac{1}{2}t,0,-t\sqrt{\frac{3}{4}-\cos^2(\pi/5)}): t \in \BR \right\rbrace $, so set
	 	$$y = \frac{1}{\sin(\pi/5)}(\frac{1}{2},0,-\sqrt{\frac{3}{4}-\cos^2(\pi/5)}).$$
	 	(See figure \ref{fig: f.d. icosahedron}, $z$ is represented by the red dot, $y$ is the blue dot and  the red triangle represents $x$.)
	 	
	 	\begin{figure}\label{fig: f.d. icosahedron}
	 		
	 		\includegraphics[width=3.5cm, height=2.5cm]{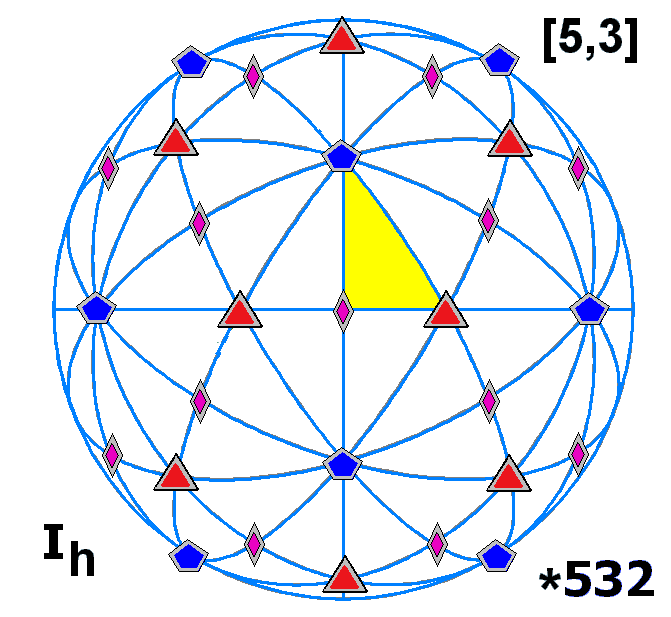}
	 		\caption{The icosahedral group. A fundamental domain and symmetries. (Figure taken from Wikipedia.)}
	 	\end{figure}
	 	
	 \end{enumerate}

	 We summarize the above in the following table (for later use we calculated also the midpoints of every side):
\begin{table}[h]
	 \begin{tabular}{|c|c|c|p{7.5cm}|}
	 	\hline
	 	& $A_3$ & $B_3$ & $H_3$  \\
	 	\hline
	 	\(z\)&$(1,0,0)$  &$(1,0,0)$  &$(1,0,0)$  \\
	 	\hline
	 	\(x\)&$ (\frac{1}{\sqrt{3}}, -\frac{2}{\sqrt{6}},0)$  & $(\frac{\sqrt{2}}{\sqrt{3}}, -\frac{1}{\sqrt{3}},0)$ &$  \frac{2}{\sqrt{3}}(\cos(\pi/5),-\sqrt{\frac{3}{4}-\cos^2(\pi/5)},0)$ \\
	 	\hline
	 	\(y\)&$(\frac{1}{\sqrt{3}},0, -\frac{2}{\sqrt{6}})$  & $ (\frac{1}{\sqrt{2}},0, -\frac{1}{\sqrt{2}})$& $\frac{1}{\sin(\pi/5)}(\frac{1}{2},0,-\sqrt{\frac{3}{4}-\cos^2(\pi/5)})$ \\
	 	\hline
	 	$m_1=\frac{x+z}{2}$ &$ (\frac{1+\sqrt{3}}{2\sqrt{3}}, -\frac{1}{\sqrt{6}},0)$ & $(\frac{\sqrt{2}+\sqrt{3}}{2\sqrt{3}}, -\frac{1}{2\sqrt{3}},0)$  &$  \frac{1}{\sqrt{3}}(\cos(\pi/5)+\frac{\sqrt{3}}{2},-\sqrt{\frac{3}{4}-\cos^2(\pi/5)},0)$ \\
	 	\hline
	 		$m_2 =\frac{x+y}{2}$ &$ (\frac{1}{\sqrt{3}}, -\frac{1}{\sqrt{6}},-\frac{1}{\sqrt{6}})$  &$(\frac{2+\sqrt{3}}{2\sqrt{6}}, -\frac{1}{2\sqrt{3}}, -\frac{1}{2\sqrt{2}})$&
	 	 $ \frac{\ba}{\sqrt{3}}\left(\frac{\sqrt{3}\si+4\co}{4\ba} ,-1,-\frac{\sqrt{3}}{2\si} \right)  $ \\
	 	
	 	\hline
	 	$m_3=\frac{z+y}{2}$ &$(\frac{1+\sqrt{3}}{2\sqrt{3}},0, -\frac{1}{\sqrt{6}})$  &$ (\frac{1+\sqrt{2}}{2\sqrt{2}},0, -\frac{1}{2\sqrt{2}})$ &$\frac{1}{2\sin(\pi/5)}(\frac{1}{2}+{\sin(\pi/5)},0,-\sqrt{\frac{3}{4}-\cos^2(\pi/5)})$ \\
	 	\hline

	 \end{tabular}
\label{tab: spherical triangles}
\caption{Vertices and segment midpoints calculated for the spherical triangles.}

	 \end{table}

	 \subsection {The optimal sphere approximation of a group}
	 We will now show that for a fixed $n$ the Coxeter group $[3,n]$ has an optimal sphere approximation. Note that if two points $x,y \in \BR^3$ are in the same side of a plane $\R$, then  reflecting $x$ about this plane will enlarge the distance from $y$. The next few lemmas will show that this phenomena happens for all the elements of $G$. Fix $n \in \left\lbrace 3,4,5\right\rbrace $ and set $G=[3,n]$.

	 \begin{lemma}\label{lem: reflections move appart}
	 	Let \(R\) be a reflection about the plane $\R$. Let $\Rperp{ }$ be a normal to \(\R\), and suppose that $x,y \in \BR^3$ are two points with
	 	$$\langle x,\Rperp{}\rangle \langle y,\Rperp{} \rangle \geq 0 $$
	 	then,
	 	$$d(x,y) \leq d(Rx,y). $$
	 \end{lemma}
	 \begin{proof}
	 	Note that for two points $x,y \in \BR^3$, the distance $d(x,y)^2$ is given by
	 	$$ d(x,y)^2 = \left\| x-y \right\|^2=\langle x-y, x-y \rangle =\left\| x\right\| ^2 +\left\| y\right\| ^2- 2\langle x,y\rangle. $$
	 	Denote by $P_{\R}x$ as the orthogonal projection of $x$ on $\R$ and $P_{\Rperp{}} x$, as the orthogonal projection of $x$ on $\R^\perp$. Then $ \langle P_{\Rperp{}} x, P_{\Rperp{}} y\rangle \geq 0$. It follows that
	 	$$d(Rx,y)^2 =  \left\| R x\right\| ^2 +\left\| y\right\| ^2-2\langle Rx,y \rangle =  \left\|x\right\| ^2 +\left\| y\right\| ^2-2[\langle RP_{\R}x,P_{\R}y \rangle +\langle RP_{\Rperp{}} x,P_{\Rperp{}} y \rangle ] =$$
	 	$$ \left\|x\right\| ^2 +\left\| y\right\| ^2-2[\langle P_{\R}x,P_{\R}y \rangle - \langle P_{\Rperp{}} x,P_{\Rperp{}} y \rangle ] =  \left\|x-y\right\| ^2 +2\langle P_{\Rperp{}} x,P_{\Rperp{}} y \rangle \geq  d(x,y)^2
	 	.$$
	 \end{proof}

\begin{lemma} \label{lem: reflections by R_2 stay on same R_1 side 2}
	Let $R_1,R_2$ be two reflection with angle $\sphericalangle(\R_1,\R_2) = \varphi\leq \pi/2$ and let $x$ by a point with   $\sphericalangle(x,\R_i) \leq \varphi$  (for $i=1,2$).
	Then
	$$ \langle x,\Rperp{1} \rangle \langle R_2 x ,\Rperp{1}\rangle \geq 0.$$
\end{lemma}
\begin{proof}
Indeed if $x \in \R_1$ then $ \langle x,\Rperp{1} \rangle =0.$ Otherwise $0<\sphericalangle(R_2x,\R_1) < \pi$ hence for every point $p$ in the segment $[x,R_2x]$, $0<\sphericalangle(p,\R_1)< \pi$, so the segment $[x,R_2x]$ has no $R_1$ fixed point. If we had
$$ \langle x,\Rperp{1} \rangle \langle R_2 x ,\Rperp{1}\rangle < 0$$
then by the intermediate value theorem we had a point $p\in [x,R_2x]$ that was fixed by $R_1$.

\end{proof}

We now show that rotations also enlarge distances.
\begin{lemma} \label {lem: rotatations of any fundemantal domain}
Let $R_1,R_2$ be two reflection with angle $\sphericalangle(\R_1,\R_2) = \varphi \leq \pi/2$ and let $x,y$ by two points with $\sphericalangle(x,\R_i) \leq \varphi$ and $\sphericalangle(y,\R_i) \leq \varphi$ (for $i=1,2$). Set $S=R_1R_2$, then
$$ d(Sx,y) \geq d(x,y)$$
\end{lemma}
\begin{proof}
Indeed,
$$d(Sx,y)=d(R_iR_jx,y)=d(R_jx,R_iy) .$$
Now $\langle R_iy,\Rperp{j}\rangle\langle x,\Rperp{j} \rangle \geq 0$ by Lemma \ref{lem: reflections by R_2 stay on same R_1 side 2} it follows from Lemma \ref{lem: reflections move appart} (applied twice) that
$$d(Sx,y) =d(R_jx,R_iy) \geq d(x,R_iy)\geq d(x,y) $$.

\end{proof}
Tsachik Gelander suggested the following Lemma which makes proofs much easier.
 \begin{lemma}\label{lem: images of fundemental domain stay on one side of reflections}
 	Let $x,y \in B$ two points in the fundamental domain, and $g \in G$ be any element. Then for every generating reflection $R_i$,
 	$$\langle gx,\Rperp{i}\rangle \langle gy,\Rperp{i}\rangle \geq 0.$$
\end{lemma}
\begin{proof}
	Assume first that $x,y$ are inner points. Suppose towards contradiction that $\langle gx,\Rperp{i}\rangle \langle gy,\Rperp{i}
	\rangle \leq 0$. Since inner product is a continuous function there is a point $z \in [gx,gy]$ with $\langle z,\Rperp{i}\rangle = 0$. Then $g^{-1}z$ is in $B$ and is fixed by $R_ig$ which is a contradiction since $B$ is a fundamental domain.
	
	If one (or both) points are at the boundary then the lemma follows by the continuity of the inner product.
\end{proof}
\begin{cor}\label{cor: conjugations of reflection also move apart}
	Let $x,y \in B$ and $g$ be any element of $G$ then $d(R_igx,gy)\geq d(x,y)$
\end{cor}
\begin{proof}
	Follows from Lemma \ref{lem: images of fundemental domain stay on one side of reflections} and Lemma \ref{lem: reflections move appart}.
\end{proof}

 \begin{lemma} \label{lem: vetrices groups}
 	Let $R_1,R_2$ be two reflections whose angle is $\pi/n$ for some $n \in \left\lbrace 2,3,4,5\right\rbrace $. Let \(H\) be the vertex group generated by $R_1,R_2$. Let $B$ the fundamental domain for the action of $H$, enclosed by the planes associated with $R_1,R_2$. Then for any $x,y \in B$, and $g \in H$
 	$$d(x,y) \leq d(gx,y). $$
 \end{lemma}
 \begin{proof}
 	Let $\R_i$, $i=1,2$ denote the reflecting planes of $R_i$ respectively and let $U=\R_1 \cap \R_2$. Let \(H'\) be the group generated by $S=R_1R_2$. Then \(H'\) is a cyclic group of rotations about the axis $U$. Note that $[H:H']=2$ , and that $\left| H\right| = n$.
 	
 	 Assume first that $g \in H'$, so $g = S^k$ for some $1 \leq k \leq n$. If $k \leq n/2$ then $g$ is a rotation at angle $k\pi/n$. It is a product of $R_1$ and an image of $R_2$, $R'$. Note that since $k \leq n/2$,  $\sphericalangle(R_1,R') \leq \pi/2$  we have $d(x,y)\leq d(gx,y)$ by Lemma \ref{lem: rotatations of any fundemantal domain}. The other case is similar. If $k>n/2$, then $g=S^{-(n-k)}$ and $g$ is a rotation which is a product of $R_2$ and an image of $R_1$, $R'$. Note that $\sphericalangle(R_2,R') \leq \pi/2$. We have $d(x,y)\leq d(gx,y)$ by Lemma \ref{lem: rotatations of any fundemantal domain}.

 	  Let $g \in H \setminus  H'$ then $g=S^kR_1$ for some element $k$.  We distinguish between two cases.
 	\begin{enumerate}
 		\item Assume first that $k=2l$. Note that $S^lR_1= R_1R_2R_1 \ldots R_1R_2R_1 = R_1S^{-l}$. Therefore,
 		$$d(gx,y) = d(R_1S^{-l}x,S^{-l}y).  $$
 	
 		The result follows from Corollary \ref{cor: conjugations of reflection also move apart} applied with $g=S^{-l}$.
 		\item Now Assume that $k=2l+1$ then,
 		$$d(gx,y) = d(R_2R_1S^{l}x,R_1S^ly).$$
 	
 			The result follows from Corollary \ref{cor: conjugations of reflection also move apart} applied with $g=R_1S^{l}$.
 	\end{enumerate}
 \end{proof}
Recall that for a given set of points $A \subset X$ in a metric space $X$, we can define Dirichlet domains related to $A$. The points in $A$ are called seeds. Every point $x \in X$ is labeled by the seed that is closest to $x$ (in $A$). This partition of $X$ is called a Voronoi diagram and every cell is called a Dirichlet domain. We now show that the fundamental domain is a Dirichlet domain related to the orbit of any point in $B$, and the set of copies of $B$ gives a tessellation of $\BR^3$ which is a Voronoi diagram.
\begin{prop} \label{prop: fundemantal domain is a Dirichlet domain}
	Let $g\in G$ be any element. Then, for every $x,y \in B$,
	$$d(x,y) \leq d(gx,y).$$
	Equivalently, $x$ is the nearest point to $y$ in $Gx$.
\end{prop}
\begin{proof}
	If $g\in G$ is any element then $g=h^{-1}g'h$ for some $g'$ in one of the vertex groups. If $g' \in G'$ is a rotation then by Lemma \ref {lem: rotatations of any fundemantal domain}
	$$d(gx,y)=d(g'hx,hy)\geq d(x,y).$$
	Otherwise as in Lemma \ref{lem: vetrices groups}, $g$ is conjugated to some $R_i$ and the Lemma follows from Corollary \ref{cor: conjugations of reflection also move apart}.
\end{proof}
\begin{cor}
	
	The Chebychev center $c$ is an optimal $G$ sphere approximation.
\end{cor}
\begin{proof}
	Recall that an optimal $G$ sphere approximation is such that
	$$d_H(Gc,S^2) = \inf_{x\in \BR^3}\sup_{y \in B'}\left\lbrace d_H(Gx,y)\right\rbrace. $$
	By Proposition \ref{prop: fundemantal domain is a Dirichlet domain},
	$$ d_H(Gx,y)= d(x,y).$$
	So
	$$d_H(Gc,S^2) = \inf_{x\in \BR^3}\sup_{y \in B'}{d(x,y)}$$
	On the other hand, for every $x$,
	$$\sup_{y \in B'}\left\lbrace d(x,y)\right\rbrace$$
	is the radius of the smallest ball centered at $x$ and containing $B'$ and the proof follows.

\end{proof}


\subsection{The Chebychev center}
Let \(G\) be one of the Coxeter groups, $A_3,B_3$ and $H_3$. Let \(B'\) be the spherical triangle which is a fundamental domain for the action of \(G\) on the sphere. Our goal now is to show that the minimal ball containing its vertices contains \(B'\). In light of Example \ref{exam: Chebychev center of triangle} this will make the task of finding the center fairly easy.

\begin{lemma}
	Let \(G\) be one of the groups, $A_3,B_3$ and $H_3$. Let $B$ and $B'$ be its fundamental domains for the action on $\BR^3$ and the sphere respectively. Set $\left\lbrace x,y,z \right\rbrace $ as  the vertices of the spherical triangle defining $B'$. Denote $c$ as the Chebychev center of the triple $\{x,y,z\}$. Then $c$ is also the Chebychev center of $B'$ .
\end{lemma}

\begin{proof}
	Clearly it is enough to show that for any two vertices, say $x,y \in V^{R_1}$, and any point $p$ in the spherical segment $[x,y]^s$ related to the segment $[x,y]$,
	$$d(c,p) \leq d(c,x).$$
	Denote the orthogonal projection of $c$ on $V^{R_1}$ by $c'$. For any point $q \in V^{R_1}$ the Pythagorean theorem implies:
	$$ d^2(q,c) = d^2(q,c')+d^2(c,c').$$
	It is therefore enough to show that,
	$$d(p,c') \leq d(x,c').$$
	Note that it follows from Table \ref{tab: spherical triangles} that for all three groups and all segments $[a,b]$ of the spherical triangles we have that
	$$\norm {m_i} \geq \norm{c} \geq \norm{c'},$$ in particular, we can assume without any loss of generality that $\sphericalangle([x,0],[0,c'])\leq \pi/2$ (otherwise argue for $y$) we obtain,
	\begin{equation} \label{eq: formula for angle }
	0  \leq \sphericalangle([0,p],[0,c'])\leq \sphericalangle([0,x],[0,c']) \leq \pi/2.
	\end{equation}
	
	set $a=\norm {c'}.$  It follows from the theorem of cosine with respect to the triangle $x,0,c'$ that
	$$ d^2(x,c') = 1+a^2-2a\cos(\sphericalangle([x,0],[0,c']))$$
	Implementing the same theorem for the triangle $p,0,c'$ we get
	$$ d^2(p,c') = 1+a^2-2a\cos(\sphericalangle([p,0],[0,c'])).$$
	The result then follows from \ref{eq: formula for angle }.
\end{proof}

 \subsection{The Chebychev center related to each of the groups } Now we find the Chebychev centers of all three fundamental domains. The center whose distance to any (all) vertex of the spherical triangle is the shortest among all three groups, will give the optimal finite homogeneous approximation. To this end,
we find the perpendicular bisectors of each of the sides in the three cases. Denote $u=x-y$ and $v=x-z$. Then the perpendicular bisectors are in the affine subspace spanned by $u,v$. Let $V$ be the space spanned by $u,v$, then an orthogonal vector to $u$ in $V$ is of the form $w_1=u+av$,
$$ \langle u+av,u \rangle =0.$$
Therefore,
$$ a= - \frac{\norm{u}}{\langle u,v\rangle}.$$
Therefore the line that pass through the midpoint $m_1=\frac{x+y}{2}$, is
$$l_1=\left\lbrace m_1+tw_1: t \in \BR \right\rbrace =\left\lbrace m_1+t(u-\frac{\norm{u}}{\langle u,v\rangle}v): t \in \BR \right\rbrace $$
Similarly the line that pass through $m_2 = \frac{x+z}{2}$ is,
$$l_2=\left\lbrace m_2+tw_2: t \in \BR \right\rbrace =\left\lbrace m_2+t(v-\frac{\norm{v}}{\langle u,v\rangle}u): t \in \BR \right\rbrace.$$
Both line intersects at the point the satisfies,
$$  m_1+t(u-\frac{\norm{u}}{\langle u,v\rangle}v)= m_2+s(v-\frac{\norm{v}}{\langle u,v\rangle}u)$$
or equivalently
\begin{equation} \label {eq: find Chebychev center}
 m_1-m_2=s(v-\frac{\norm{v}}{\langle u,v\rangle}u)-t(u-\frac{\norm{u}}{\langle u,v\rangle}v).
\end{equation}

We get a set of $3$ linear non homogeneous
equations which by the way constructed has a unique solution. We solved the equations for each group using matlab\textsuperscript{\textregistered}. The distance $d_H(G,S^2)$ will be then $d(c,x)$ where $c$ is the Chebychev center. We summarize the results in the following table.

\begin{table}[h!]

\begin{tabular}{|c|c|c|c|}
	\hline
	Group& $(t,s)$& Chebychev center & $d_H(G,S^2)$\\
	\hline
   $A_3$ &$(0.3255 ,0.0872)$ & $(0.6511,-0.3370,-0.3370)$&$0.5907$\\
\hline
$B_3$ &$(0.3929, 0.0397)$ & $(0.7858,-0.2498,-0.3255)$&$0.4628$\\
\hline

$H_3$ &$( 0.4485, 0.0153)$ & $(0.8971,-0.1655,-0.2548)$&$0.3208$\\	
\hline	
\end{tabular}
\label{tab: Chebychev centers}
\caption{Three Chebychev centers. The groups names are according to the Coxeter notation, $(t,s)$ is the solution to Equation \ref{eq: find Chebychev center} and $D_H(G,S^2)=d(c,x)=d(c,y)=d(c,z)$. }
\end{table}
 \begin{example}
 	We give now some examples of some known polyhedrons.
 	\begin{enumerate}
 		\item \textbf{The octahedron and the cube}. The octahedron is the orbit of $y$ under the octahedral group $B_3$. Note that $d(x,y) = 0.9194$ (clearly $d(y,z)$ is smaller). The Hausdorff distance from the orbit of $y$ to the $2$ dimensional sphere is $0.9194$. For every $t>0$ the orbit of $ty$ will also be an octahedron and all octahedrons arise in this way. One can easily compute that the optimal $t$ is $t=0.5774$. Then the distance is $d(G(ty),S^2) = 0.8165$. The cube is its dual. It is the orbit of $x$ (and any $tx$) and has the same Hausdorff distance from the sphere.
 		\item \textbf{The icosahedron and the dodecahedron}. Both Platonic solids are homogeneous spaces of $H_3$. The first is the orbit of $y$. Its stabilizer is therefore $G_2 = \langle R_2,R_3\rangle$. This group has 10 elements hence the orbit has 12 vertices and its faces are all regular triangles. The second is its dual. It is the orbit of $x$. It has 20 vertices and its faces are all regular pentagons. The distance $d(0.7947x,y) = 0.6071$ and this is also the Hausdorff distance of both polyhedrons from the sphere.
 		\item \textbf{The truncated icosahedron.} The truncated icosahedron for example is the polyhedron whose vertices are the $H_3$ orbit of the point that lies in the middle of $[y,z]$, which we denote by $m_3$. The reflection $R_3$ is its stabilizer and the action of $G_2 = \langle R_2,R_3\rangle$ has 5 points forming a regular pentagon. On the other hand the orbit group $G_1=\langle R_1,R_2 \rangle$  has 6 points which generate a regular hexagon. This homogeneous space is very familiar as it was the one skeleton of the official soccer ball for many years. The distance $d(m_3,x)=0.443$ and this is also the Hausdorff distance to the $2$ dimensional sphere. It is interesting to note that the orbit of $0.9945m_2$- the middle of $[x,y]$ is of Hausdorff distance $0.3354$. The orbit is the vertices set of the Archimedean solid, rhombicosidodecahedron. Indeed $m_2$ is very close to the Chebychev center (we had $s=0.0153$) therefore its orbit is very close to be optimal $G$ sphere approximation.

\item  In the table below we list all 13 Archimedean and five Platonic solids and their distance from the sphere. It is pleasant to mention here the work of Shaked Bader and Sapir Freizeit. They calculated these distances in an exercise  assigned to them by Tsachik Gelander. Note that if a solid is the orbit of some point $p$ then for every $t>0$, the orbit of $tp$ will give the same solid scaled differently. We performed grid search with matlab to get best approximation.
  	\end{enumerate}
 \end{example}

\begin{tabular}{|c|c|c|c|}
	\hline
	Solid	& generating point& isometry group& distance to the sphere \\
		\hline
\multicolumn{4}{|c|}{Platonic solids}\\
\hline
Tetrahedron & $1/3x,1/3y$&$A_3$& $0.9428$\\

\hline

Octahedron & $0.5774y$ &$B_3$& $0.8165$\\

\hline
cube&$0.5774x$&$B_3$& $0.8165 $\\
\hline
		Icosahedron& $0.7947y$ &$H_3$ &$0.6071$\\
\hline
	Dodecahedron& $0.7947x$ &$H_3$ &$0.6071$\\
\hline
\multicolumn{4}{|c|}{Archimedean solids}\\
\hline
Truncated Tetrahedron &$0.5774m_1,0.5774m_3$
&$A_3$&$0.8586$\\
\hline

Cuboctahedron &$0.8660m_2$
&$A_3$&$0.7071$\\
\hline
Truncated Cube & $0.7071m_1$& $B_3$ & $0.7388$\\
\hline
Truncated Octahedron &$m_3$& $B_3$ &$ 0.678$\\
\hline

Rhombicuboctahedron & $0.9659m_2$& $B_3$ & $0.5140$\\
\hline
Truncated Cuboctahedron & $0.9516(x+y+z)/3$&$B_3$&0.5248\\
\hline
Snub Cube &  $0.9516(x+y+z)/3$& $B_3$ &0.5248 \\
\hline
Icosidodecahedron & $0.8507z$ &$H_3$ &0.5257\\
\hline
Truncated Dodecahedron& $0.8507m_1$ &$H_3$ &0.5479\\
\hline
Truncated Icosahedron& $m_3$ & $H_3$& 0.443\\
\hline
Rhombicosidodecahedron& $09945m_2$ &$H_3$ &0.3354\\
\hline
Truncated Icosidodecahedron& $0.9727(x+y+z)/3$ &$H_3$ &0.3773\\
\hline
snub dodecahedron& $0.9727(x+y+z)/3$ &$H_3$ &0.3773\\
\hline
\end{tabular}

To conclude, the optimal finite homogeneous approximation of the $S^2$ sphere is a space with 120 elements which are the orbit of the Chebychev center of the fundamental domain of the group $H_3$. Like the Truncated icosidodecahedron its faces are squares, regular hexagons and regular decagons.
%
%

\bibliographystyle{plain}
\bibliography{soccer}

\begin{thebibliography}{1}

\bibitem{bekka2008kazhdan}
Bachir Bekka, Pierre de~La~Harpe, and Alain Valette.
\newblock {\em Kazhdan's property (T)}, volume~11.
\newblock Cambridge university press, 2008.

\bibitem{benjamini2012scaling}
Itai Benjamini, Hilary Finucane, and Romain Tessera.
\newblock On the scaling limit of finite vertex transitive graphs with large
  diameter.
\newblock {\em arXiv preprint arXiv:1203.5624}, 2012.

\bibitem{coxeter1940regular}
HSM Coxeter.
\newblock Regular and semi-regular polytopes. i.
\newblock {\em Mathematische Zeitschrift}, 46(1):380--407, 1940.

\bibitem{gelander2012limits}
Tsachik Gelander.
\newblock Limits of finite homogeneous metric spaces.
\newblock {\em arXiv preprint arXiv:1205.6553}, 2012.

\end{thebibliography}

\end{document}